\documentclass[11pt, leqno]{amsart}

\usepackage{graphicx}
\usepackage[colorlinks,citecolor=red,pagebackref,hypertexnames=false]{hyperref}
\usepackage{amsfonts,slantsc}
\usepackage{amsmath,amssymb,amsthm,amssymb,amscd,cancel,stackrel}
\usepackage{blkarray}
\usepackage{multirow, mathtools,leftidx}
\usepackage{memhfixc}
\usepackage{latexsym}
\usepackage{tikz}
\usepackage{color}
\usepackage{enumerate}
\usepackage{relsize}

\newcommand{\fg}{\mathfrak g}

\newcommand{\ZZ}{\mathbb{Z}}

\newcommand\vect[2]{#1_1,\,\ldots,\, #1_{#2}}
\newcommand{\dI}{{}^{\delta}\!I}
\newcommand{\dL}{{}^{\delta}\!{{\mathcal L}}}
\newcommand{\gJ}{{}^{\gamma}\!J}

\DeclareMathOperator{\Ann}{Ann}
\DeclareMathOperator{\Char}{char}

\DeclareMathOperator{\Hom}{Hom}
\DeclareMathOperator{\Hf}{Hilb}

\DeclareMathOperator{\Ker}{Ker}

\DeclareMathOperator{\rank}{rank}

\DeclareMathOperator{\Span}{Span}

\DeclareMathOperator{\lpp}{\mathcal{L}}    
\newcommand{\ci}{\mathfrak{f}}    
\DeclareMathOperator{\sec2}{\underline{\mathbf{2}}}  
\DeclareMathOperator{\seca}{\underline{\mathbf{a}}}  
\DeclareMathOperator{\EGH}{EGH} 

\newtheorem{theorem}{Theorem}[section]
\newtheorem{lemma}[theorem]{Lemma}
\newtheorem{proposition}[theorem]{Proposition}
\newtheorem{corollary}[theorem]{Corollary}
\newtheorem{conjecture}[theorem]{Conjecture}
\newtheorem{main-conjecture}[theorem]{Main Conjecture}
\newtheorem{claim}[theorem]{Claim}

\newtheorem{lem-def}[theorem]{Lemma and Definition}
\newtheorem{prop-def}[theorem]{Proposition and Definition}
\newtheorem{observe/question}[theorem]{Observation/Question}

\theoremstyle{definition}
\newtheorem{definition}[theorem]{Definition} 
\newtheorem{remark}[theorem]{Remark}
\newtheorem{notation}[theorem]{Notation}
\newtheorem{rem-def}[theorem]{Remark and Definition}
\newtheorem{example}[theorem]{Example}

\newtheorem{mquestion}[theorem]{Main Question}

\newcount\HOUR
\newcount\MINUTE
\newcount\HOURSINMINUTES
\newcount\INTVAL
\newcommand{\twodigit}[1]{\INTVAL=#1\relax\ifnum\INTVAL<10 0\fi\the\INTVAL}
\HOUR=\time\divide\HOUR by 60\relax
\HOURSINMINUTES=\HOUR\multiply\HOURSINMINUTES by 60\relax
\MINUTE=\time\advance\MINUTE by -\HOURSINMINUTES\relax
\newcommand\rightnow{
            \twodigit{\number\day}.\space
            \ifcase\month\or January\or February\or March\or April\or
May\or June\or July\or August\or September\or October\or November\or
December\fi
            \space\number\year}

\title[The $\EGH$ conjecture for defect two quadratic ideals]{The Eisenbud-Green-Harris Conjecture for \\ Defect Two Quadratic Ideals} 

\author{Sema G\"unt\"urk\"un}
\address{Department of Mathematics, University of Michigan, 
530 Church Street East Hall, MI, 48109  \\
(Current Address: Department of Mathematics, University of Connecticut, 341 Mansfield Road, CT, 06269)}
\email{gunturku@umich.edu}
\author{Melvin Hochster}
\address{Department of Mathematics, University of Michigan, 
530 Church Street East Hall, MI, 48109}
\email{hochster@umich.edu}


\begin{document}



\begin{abstract} The Eisenbud-Green-Harris (EGH) conjecture states that a homogeneous ideal in a polynomial ring $K[x_1,\,\ldots,\,x_n]$ over a field $K$ that contains a regular sequence $f_1,\,\ldots,\, f_n$ with degrees $a_i$, $i=1,\,\ldots,\,n$ has the same Hilbert function as a lex-plus-powers ideal containing the powers $x_i^{a_i}$, $i=1,\,\ldots,\,n$.  In this paper, we discuss a case of the EGH conjecture for homogeneous ideals generated by $n+2$ quadrics containing a regular sequence 
$f_1,\, \ldots, \, f_n$ and give a complete proof for EGH when $n=5$ and $a_1=\cdots=a_5=2$.\end{abstract}

\maketitle 

\section{Introduction}  
Let $R = K[x_1,\,\ldots,\,x_n]$ be the polynomial ring in $n$ variables over a field $K$ with the homogeneous 
lexicographic order  in which $x_1> \cdots > x_n$ and with the standard grading $R= \bigoplus\limits_{i\geq 0} R_i$.
We denote the Hilbert function of a $\ZZ$-graded $R$-module $M$
by $\Hf_{M}(i) := \dim_K M_i$, where $M_i$ is the homogeneous component of $M$ in degree $i$. When $I$ is
a homogeneous ideal of $R$ and $M$ is $R$, or $I$,  or  $R/I$,  the Hilbert function has value 0 when $i <0$.  When the 
Hilbert function of $M$ is 0 in negative degree, we may discuss the Hilbert function of $M$ by giving the sequence
of its values, and we refer to this sequence of integers as the \textit{O-sequence} of $M$. 

In 1927, Macaulay \cite{Ma} showed that the Hilbert function of any homogeneous ideal of $R$ is attained by a lexicographic ideal in $R$. Later, in Kruskal-Katona's theorem \cite{Ka,Kr}, it is shown that the polynomial ring $R$ in Macaulay's result can be replaced with the quotient $R/(x_1^2,\,\ldots,\,x_n^2)$.  After this result, Clement and Lindstr\"om, in \cite{CL}, generalized the result to $R/(x_1^{a_1},\,\ldots,\,x_n^{a_n})$ if $a_1 \leq \cdots \leq a_n < \infty$.  

In \cite{EGH} Eisenbud, Green and Harris conjectured a generalization of the Clement-Lindstr\"om result.  
Let ${\seca}=(a_1,\,\ldots,\,a_n) \in \mathbb{N}^n$, where $2\leq a_1 \leq \ldots \leq a_n$. 

\begin{conjecture}[Eisenbud-Green-Harris ($\EGH_{\seca,n}$) Conjecture \cite{EGH}] If $I$ is a homogeneous ideal in $R=K[x_1,\,\ldots,\,x_n]$ containing a regular sequence $f_1,\, f_2,\,\ldots,\,f_n$ with degrees 
$\deg f_i = a_i$, then there is a monomial ideal $\lpp = (x_1^{a_1}, \,\ldots,\,x_n^{a_n})+J$, where $J$ is a lexicographic ideal in $R$, such that $R/\lpp$ and $R/I$ 
have the same Hilbert function.
\end{conjecture}

Although there has been some progress on the conjecture, it remains open. The conjecture is shown to be true for $n=2$ by Richert in \cite{Ri}. Francisco \cite{Fr} shows the conjecture for almost complete intersections. Caviglia and Maclagan in \cite{Cav-Mac}  prove the result if $a_i > \sum\limits_{j=1}^{i-1}(a_j-1)$ for $2\leq i \leq n$. 
The rapid growth required for the degrees does not yield much insight into cases like the one in which the regular sequence consists of quadratic forms. When $n=3$, Cooper in \cite{Cooper} proves the EGH conjecture for the cases where  
$(a_1, a_2, a_3) = (2,a_2,a_3)$ and $(a_1, a_2, a_3) = (3,a_2,a_3)$ with $a_2 \leq a_3 \leq a_2+1$. 

One of the most intriguing cases is when $a_1=\cdots = a_n = 2$ for any $n\geq 2$, which is the case for which Eisenbud, Green and Harris originally stated their conjecture.   It is known that the conjecture holds for homogeneous ideals minimally generated by generic quadrics: the case where $\Char K = 0$  was proved by Herzog and Popescu \cite{HP}  and the case of  arbitrary characteristic was proved by Gasharov \cite{Gasharov} around the same time. There have been several other results on the EGH conjecture. More recently, the case when every $f_i$, $i=1,\,\ldots,\,n,$ in the regular sequence is a product of linear forms is settled by Abedelfatah in \cite{Abed}, and results on the EGH conjecture using linkage theory are given by Chong \cite{Chong}.

In this paper we focus on the case when the degrees of the elements of the regular sequence are $a_1= \cdots =a_n=2$. In \cite{Ri}, Richert claimed that the conjecture for quadratic regular sequences is true for $2\leq n\leq 5$, but this work has not been published, and other researchers have been unable to verify this for $n=5$ thus far. Chen, in \cite{Chen}, has given a proof for the case where $n\leq 4$ when $a_1= \cdots=a_n=2$.

In $\S 2$ we recall some definitions and results from the papers of Francisco \cite{Fr},  Caviglia-Maclagan \cite{Cav-Mac} and Chen \cite{Chen}. In $\S 3$ we 
study homogeneous ideals $I$ generated by $n+2$ quadratic forms  in $n$ variables containing a regular sequence of length $n$, and Theorem~\ref{thm-defect2-egh(2)} shows that there is a monomial ideal 
$\lpp = (x_1^2,\,\ldots,\,x_n^2)+J$, where $J$ is a lexicographic ideal in $R$, such that $R/I$ and $R/\lpp$ have the same Hilbert function in degree $2$ and $3$ (i.e., $\EGH_{(2,\,\ldots,\,2), n}(2)$ holds: see Definition \ref{defn-EGH(d)}). In $\S 4$ we give a proof to the claim of Richert for the quadratic regular sequence case when $n=5$.

%

%
%
%
%
%
%
%
%
%
%
%
%
%
 \section{Background and Preliminaries}  
In this section we recall some definitions and state some known results that are used throughout the paper. 
\begin{definition} Let $u = x_1^{a_1}\cdots x_n^{a_n}$ and $v = x_1^{b_1}\cdots x_n^{b_n} $ be monomials in $R$ of the same degree. We say that $u$ is greater than $v$ with respect to the {\it lexicographic} (or {\it lex}) order if there exists an $i$ such that 
$a_i > b_i$ and $a_j = b_j$ for all $j < i$. 

A monomial ideal $J \subseteq R$ is called a \textit{lexicographic ideal} (or \textit{lex ideal}) if, for all degrees $d$, the $d$-th degree component of $J$, denoted by $J_d$, is spanned over the base field $K$ by an initial segment of the degree $d$ monomials in the lexicographic order.
\end{definition}
\begin{definition} Given $2\leq a_1 \leq \cdots \leq a_n$, a \textit{lex-plus-powers ideal} (LPP ideal) $\lpp$ is a monomial ideal in $R$ that can be written as $\lpp = (x_1^{a_1},\,\ldots,\,x_n^{a_n}) + J$ where $J$ is a lex ideal in $R$. 
\end{definition}
This definition agrees with the one in \cite{Cav-Mac}.  Some authors require that the $x_i^{a_i}$ be minimal generators
of $\lpp$, which we do not.  However, since we consider only nondegenerate homogeneous ideals in this paper,
i.e., ideals contained in $(x_1, \, \ldots,\,x_n)^2$, in the case where $a_1 = \cdots = a_n = 2$  it is automatic
that the $x_i^2$ are minimal generators of the ideal under consideration.

In \cite{Fr} Francisco  showed the following for almost complete intersections.
\begin{theorem}[Francisco \cite{Fr}] \label{EGH(d)-ci} Let integers $2 \leq a_1 \leq a_2 \leq \cdots \leq a_n$ and $d \geq a_1$ be given. 
Let the ideal $I$ have minimal generators $f_1,\,\ldots,\,f_n,\,g$ where $f_1,\,\ldots,\,f_n$ form a regular sequence with $\deg f_i = a_i$  and $g$ has degree $d$. Let $\lpp = (x_1^{a_1},\,\ldots,\,x_n^{a_n}, m)$ be the lex-plus-powers ideal where $m$ is the greatest monomial in lex order in degree $d$ that is not in $(x_1^{a_1},\,\ldots, \,x_n^{a_n})$. 
 Then $\Hf_{R/I}(d+1) \leq \Hf_ {R/\lpp}(d+1)$. 
\end{theorem}

Note that, necessarily, $d \leq \sum_{i=1}^n (a_i-1)$, since $(\vect f n)$ contains all forms of degree larger than that.
If $a_1 = \cdots = a_n = 2$,  then $d \leq n$.

The following corollary is an immediate consequence of Theorem \ref{EGH(d)-ci} above.  
 If  $g \in R$ is a nonzero form of degree $i$ we write $gR_j$ for
the vector space $\{gh: h \in R_j\} \subseteq R_{i+j}$.
\begin{corollary}\label{cor-intersection-dim} Let $I = (f_1,\,\ldots,\,f_n,g)$ be an almost complete intersection as in Theorem 
\ref{EGH(d)-ci} above such that $a_1 = \cdots = a_n = 2$.   Then
\[ \dim_K \big( (f_1,\,\ldots,\,f_n)_{d+1} \cap gR_1 \big) \leq d.\]
\end{corollary}

\begin{proof} We can write 
\[
\dim_K I_{d+1} = \dim_K(f_1,\,\ldots,\,f_n)_{d+1} + \dim_K gR_1 - \dim_K \big( (f_1,\,\ldots,\,f_n)_{d+1} \cap gR_1 \big),
\]
where $\dim_KgR_1 = n$.  
Then by Theorem \ref{EGH(d)-ci}, we have
\[\dim_K I_{d+1} \geq \dim_K (x_1^2,\,\ldots,\,x_n^2, x_1\cdots x_d)_{d+1}= \dim_K(x_1^2,\,\ldots,\,x_n^2)_{d+1} + n -d\]
Since $\Hf_{R/(f_1,\,\ldots,\,f_n)}(i) = \Hf_{R/(x_1^2,\,\ldots,\,x_n^2)}(i)$ for all $i\geq 0$, 
we can conclude that
\[
\dim_K \big( (f_1,\,\ldots,\,f_n)_{d+1} \cap gR_1 \big) \leq d.
\]
\end{proof}

The next statement is a weaker version of the $\EGH_{\seca,n}$ conjecture. It focuses on the Hilbert function of the given homogeneous ideal only at the two consecutive degrees $d$ and $d+1$ for some non-negative integer $d$. 
\begin{definition}[$\mathbf{\EGH_{\seca,n}(d)}$]\label{defn-EGH(d)} Following Caviglia-Maclagan \cite{Cav-Mac}, we say that ``$\EGH_{\seca,n}(d)$ holds" if for any homogeneous ideal $I \in K[x_1,\,\ldots,\,x_n]$ containing a regular sequence of degrees $\seca = (a_1,\,\ldots,\,a_n)$, where
$2\leq a_1 \leq \cdots \leq a_n$, there exists a lex-plus-powers ideal $\lpp$ containing $\{x_{i}^{a_i} : 1\leq i \leq  n\}$ such that 
\[\dim_K I_d = \dim_K \lpp_d \quad \text{and} \quad \dim_K I_{d+1} = \dim_K \lpp_{d+1}.\]
\end{definition}

\begin{lemma}\label{equal-EGH(d)} The condition $\EGH_{(d,\,\ldots,\,d), n}(d)$ on a polynomial ring
$R = K[x_1,\,\ldots,\,x_n]$
is equivalent to the statement that for the ideal $I$ generated by  $n + \delta$ $K$-linearly independent forms of 
degree $d$ containing a regular sequence of quadrics, one has that $\dim_K I_{d+1} \geq \dim_K {\lpp}_{d+1}$,  where $\lpp =
(x_1^d,\,\ldots,\,x_n^d) +  J'$ and $J'$ is minimally generated by the  greatest in lex order  $\delta$ forms of degree $d$ not already in $(x_1^d,\,\ldots,\,x_n^d)$.  
\end{lemma}

\begin{proof}  If there is an LPP ideal $(x_1^d,\,\ldots,\,x_n^d) + J$,  where $J$ is a lex ideal, with
the same Hilbert function as $I$ in degrees $d$ and $d+1$,  it is clear that $J_d$ must be spanned over $K$ 
by the specified generators of $J'$, so that $(x_1^d,\,\ldots,\,x_n^d) + J' \subseteq (x_1^d,\,\ldots,\,x_n^d) + J$,
 which implies the specified inequality on the Hilbert functions.  Moreover, when that inequality holds
 we may increase $\lpp:=(x_1^d, \, \ldots,\,x_n^d) + J'$ to an LPP ideal with the same Hilbert function as
 $I$ in degrees $d$ and $d+1$:
 if $\Delta = \Hf_{I} (d+1) - \Hf_{\lpp} (d+1)$, we may simply include the greatest (in lex order)  $\Delta$ forms of
degree $d+1$ not already in  $\lpp$.
  \end{proof}

\begin{remark} We shall eventually be focused on $\EGH_{\seca,n}(d)$ in the case where
$a_1= \cdots = a_n = d =2$, simply referred as $\EGH_{(2,\,\ldots,\,2),n}(2)$ or $\EGH_{\sec2,n}(2)$.  We shall routinely make use of this lemma in this case of quadratic regular sequence and $d=2$.
\end{remark}

\begin{lemma}[Caviglia-Maclagan \cite{Cav-Mac}] \label{CM-lemma}Fix $\seca = (a_1,\,\ldots,\,a_n) \in \mathbb{N}^n$ where $2 \leq a_1 \leq a_2 \leq \cdots \leq a_n$ and set $s = \sum\limits_{i=1}^{n}(a_i-1).$ Then for any $0\leq d \leq s-1$, $\EGH_{\seca,n}(d)$ holds if and only if $\EGH_{\seca,n}$(s-1-d) holds.
\\ 
 \noindent
Furthermore, the $\EGH_{\seca,n}$ conjecture holds if and only if $\EGH_{\seca,n}(d)$ holds for all degrees $d\geq 0$.
\end{lemma}

From now on, we always assume $\seca = \sec2 = (2,\,\ldots,\,2)$ for $n\geq 2$, unless it is stated otherwise. 
\begin{remark}\label{Rmk-EGH-01} For any $n\geq 2$, $\EGH_{\sec2,n}(0)$ holds trivially. 
In \cite[Proposition 2.1]{Chen}, Chen showed that $\EGH_{\sec2,n}(1)$ is true for any $n\geq 2$.
\end{remark}

Chen proved the following.
\begin{theorem}[Chen \cite{Chen}]
The $\EGH_{\sec2, n}$ conjecture holds when $2\leq n \leq 4$.
\end{theorem}

Chen's proof of this uses Lemma \ref{CM-lemma} above, and the observation that, when $n=4$, to demonstrate that 
the $\EGH_{\sec2,4}$ conjecture is true, it suffices to show  that $\EGH_{\sec2,4}(0)$  and $\EGH_{\sec2,4}(1)$  are 
true.

%
%
%
%
%
%
%
%
%
%
%
\section{$\EGH_{\sec2,n}(2)$ for defect two ideals} 

In this section, we focus on the homogeneous ideals in $K[x_1,\,\ldots,\,x_n]$ for $n\geq 5$ that are generated by $n+2$ quadratic forms containing a regular sequence. In particular, we study their Hilbert functions in degree $3$.
\begin{definition} If $I$ is a homogeneous ideal minimally generated by $n+\delta$ forms that contain a regular sequence of length $n$, then $I$ is said to be a \textit{defect $\delta$} ideal.
\end{definition}
Clearly, when $\delta=0$ then $I$ is generated by a regular sequence, it is a complete intersection, and we understand the Hilbert function completely.  If $\delta=1$, then $I$ is an almost complete intersection.
\begin{definition} We call a homogeneous ideal a \textit{quadratic ideal} if it is generated by quadratic forms. 

Let $I = (f_1,\,\ldots,\,f_n, g,h)$ be a homogeneous ideal minimally generated by $n+2$ quadrics where $f_1,\,\ldots,\,f_n$ form a regular sequence. We call such an ideal a \textit{defect two ideal generated by quadrics} or simply a \textit{defect two quadratic ideal}. More generally, if a quadratic ideal is a defect $\delta$ ideal, then we call it \textit{defect $\delta$} quadratic ideal.
\end{definition}
\begin{example}
The lex-plus-powers ideal $\lpp =  (x_1^2,\,\ldots,\,x_n^2, x_1x_2, x_1x_3)$ in $R$ is also a defect two quadratic ideal. 

Further, for any homogeneous defect two quadratic ideal $I$, we have the equality 
 \[\dim_K I_2 = n+2 = \dim_K \lpp_2.\]
 \end{example}
\begin{mquestion}[$\EGH_{\sec2,n}(2)$ for defect two quadratic ideals]\label{main-question} 

For any $n\geq5$, is it true that  \[\dim_K I_3 \geq n^2+2n-5 =\dim_K \lpp_3?\]	

\end{mquestion}

An affirmative answer for this question is proved completely in Theorem~\ref{thm-defect2-egh(2)} below.

\begin{notation}\label{notfg} Throughout the rest of this paper we write  $\ci$ for the ideal $(\vect f n)R$ when $\vect f n$ is a regular
sequence of quadratic forms,  and in the defect $\delta$ quadratic ideal case we write $\fg$ for the additional 
generators $g_1,\,\ldots,\,g_\delta$ of the quadratic ideal. Here, $\vect fn, \vect g \delta$ are assumed to be
linearly independent over $K$. Moreover, henceforth,
we write $J$ for the ideal $\ci + (\vect g {\delta-1})$.
 However, when $\delta = 1$ or $2$ we may write $g,\,h$ for $g_1, g_2$, so that 
whenever $\delta =2$ we henceforth write $J$ for the ideal  $\ci + (g_1) = \ci + (g)$. We denote the graded Gorenstein Artin $K$-algebra $R/\ci$ by $A$. \end{notation}

We know that, if $a_1 = \cdots = a_n = \deg g =2$, Theorem \ref{EGH(d)-ci} shows that 
$$\dim_K J_3 \geq n^2+n-2$$ and then Corollary \ref{cor-intersection-dim} gives 
$\dim_K \big(  \ci_3 \cap gR_1 \big)  \leq 2$. 


\begin{remark}\label{defect2-chen} In \cite[Proposition 3.7]{Chen} Chen gave a positive answer to the Question \ref{main-question} for defect two quadratic ideals $I = \ci +(g,h)$ if $\dim_K \big(  \ci_3 \cap gR_1 \big) = 2$.  We shall make repeated use of this fact in the sequel.
\end{remark}

In this section  we show $\EGH_{\sec2, n}(2)$ for a defect two quadratic ideal $I = \ci +(g,h)$ under the condition that 
$\dim_K \big( \ci_3 \cap g'R_1 \big) \leq 1$ for all $g' \in Kg+Kh - \{0\}$: this covers all the cases for which Chen's result in 
Proposition \ref{defect2-chen} is not applicable.\bigskip

\begin{lemma}\label{EGH(2)-cases} As in Notation~\ref{notfg},  $J$ is the defect 1 quadratic ideal $\ci + gR$.  Then: 
$$\dim_K I_3  = n^2 + 2n - \dim_K \big( \ci_3 \cap gR_1 \big)-  \dim_K \big( J_3 \cap hR_1 \big).$$

Consequently,  for the cases that are not covered by the Proposition \ref{defect2-chen} we have:\\

{\bf(i)} If $\dim_K \big( \ci_3 \cap gR_1 \big) =1$ then $\dim_K I_3 = n^2+2n-1 - \dim_K \big( J_3 \cap hR_1 \big)$, and $\EGH_{\sec2,n}(2)$ holds for a defect two quadratic ideal $I$ if and only if  $\dim_K \big( J_3 \cap hR_1 \big) \leq 4$.\\

{\bf(ii)} If $\dim_K \big( \ci_3 \cap gR_1 \big) =0$ then $\dim_K I_3 = n^2+2n - \dim_K \big( J_3 \cap hR_1 \big)$, and $\EGH_{\sec2,n}(2)$ holds for $I$  if and only if  $\dim_K \big( J_3 \cap hR_1 \big) \leq 5$. 
\end{lemma}
\begin{proof}   We have:
\begin{align*}
\nonumber \dim_K I_3  &= \dim_K J_3+ \dim_K( hR_1 )- \dim_K \big( J_3 \cap hR_1 \big) \\ 
                   &= \dim_K \ci_3 + \dim_K( gR_1 ) - \dim_K \big(  \ci_3 \cap gR_1 \big)  +\dim_K( hR_1 )- \dim_K \big( J_3 \cap hR_1 \big) \\  \nonumber
                 &= n^2 + 2n - \dim_K \big( \ci_3 \cap gR_1 \big)-  \dim_K \big( J_3 \cap hR_1 \big),
\end{align*}
and then (i) and (ii) are immediate.
\end{proof}

\begin{remark} \label{remark-for-n=5}
Let $n =5$, so that $\ci = (f_1,\,\ldots,\,f_5)$. For a defect two quadratic ideal $I=(\ci, g,h) \subseteq K[x_1,\,\ldots,\, x_5]$, if $\dim_K \big( \ci_3 \cap gR_1 \big) =0$ then clearly $\dim_K \big( (\ci,\,g)_3 \cap hR_1 \big)   \leq  \dim_K (hR_1) \leq 5$, therefore  
$\EGH_{\sec2, 5}(2)$ holds for such an ideal $I$.  However, we must give an argument to cover all possible cases, that is, when $\dim_K \big( \ci_3 \cap gR_1 \big) =1$, to be able to confirm $\EGH_{\sec2, 5}(2)$ for every defect two quadratic ideal.  In the last section, we discuss the EGH conjecture for $n=5$ and $a_1=\cdots=a_5=2$ in detail.
\end{remark}

%
Next, we proceed with two useful lemmas.
\begin{lemma}\label{lemma1} Let $A$ be the graded Gorenstein Artin $K$-algebra $R/\ci$ with $\dim_K A_1= n$. Let $g, h$ be two quadratic forms such that $gA_1= hA_1$. Then 
$\Ann_{A_1} g = \Ann_{A_1}h$.

Moreover, $\Ann_{A_i}(g)= \Ann_{A_i}(h) $ if $i\neq n-2$.
\end{lemma}

\begin{proof}  Suppose that the linear annihilator space of $g$, $\Ann_{A_1} g$, has dimension $a$ and $gA_1=hA_1$. Thus $gA_1$ has dimension $n-a$ and clearly $hA_1$ and $\Ann_{A_1} h$ have dimensions $n-a$ and $a$, respectively.

Notice that $gA(-2) \cong A /\Ann_A(g)$, hence it is Gorenstein and it has a symmetric O-sequence
\[(0, 0, 1, n-a, e_4, e_5,\,\ldots,\,e_5, e_4, n-a, 1),\]
 where $e_i$ denotes the dimension of $[gA]_i$ and $e_i = e_{n-i+2}$ for $2\leq i \leq n$. 
Then the Hilbert function of $A/gA$ is 
\[(1,  n,  \binom{n}{2}-1, \binom{n}{3}-n+a, \binom{n}{4}-e_4,\,\ldots,\,\binom{n}{3}-e_5, \binom{n}{2}-e_4, a, 0 ).\]

Since $\Ann_A(g) \cong \Hom_K( A / gA, A) \cong (A/gA)^\vee $, the Hilbert function of  $\Ann_A(g)$ is 
\[ (0, a, \binom{n}{2}-e_4,\,\ldots,\,\binom{n}{4}-e_4, \binom{n}{3}-n+a, \binom{n}{2}-1, n, 1).\] 

Recall that $gA_1 = hA_1$, $gA_i = hA_i$ for all $i\geq 2$, so $(g, h)A$ has the Hilbert function 
\[(0, 0, 2, \underbrace{n-a, e_4,\,\ldots,\,e_4, n-a, 1}_{\text{ the same as for } gA}).\]

Then the O-sequence of $A/(g,h)$ becomes
\[(1, n, \binom{n}{2}-2, \binom{n}{3}-n+a, \binom{n}{4}-e_4,\,\ldots,\,\binom{n}{3}-e_5, \binom{n}{2}-e_4, a, 0 ),\]

and it follows that $\Ann_A (g,h)$ has the Hilbert function
\[(0, a, \binom{n}{2}-e_4,\,\ldots,\,\binom{n}{4}-e_4, \binom{n}{3}-n+a, \binom{n}{2}-2, n, 1).\]
We know that $\Ann_A (g,h) = \Ann_A (g) \cap \Ann_A (h)$, and in degree $1$, $\Ann_A(g,h)$ has dimension $a$, so 
$\Ann_{A}(g,h) = \Ann_{A_1}(g) = \Ann_{A_1}(h)$. Further, $\Ann_A(g)$ and $\Ann_A(h)$ are the same in every degrees except in degree $n-2$. 
\end{proof}

\begin{lemma}\label{lemma2} Let $g, h$ be two quadratic forms in a graded Gorenstein Artin $K$-algebra $A$ such that $gA_i = hA_i$ and $g, h$ have the same annihilator space $V$ in $A_i$ for some $i\geq 1$. Then there exists $g' \in Kg+Kh - \{ 0\}$ such that \[ \dim_K \Ann_{A_i}(g') \geq \dim_K V +1.\]
\end{lemma}

\begin{proof} Consider the multiplication maps by $g$ and $h$,
\[\phi_g : A_i/V \to  gA_i \, \, \, \text{ and }   \, \, \,  \phi_h : A_i/V\to  hA_i\]
whose images $gA_i$, $hA_i$ are subspaces in $A_{i+2}$ and $gA_i=hA_i$ by assumption. Then there is a automorphism 
\[ T : A_i/V \to A_i/V\]
 such that  $g\ell = hT(\ell)$ for any $\ell\in A_i/V$. However, $T$ has at least one nonzero eigenvector $u$ with $T(u) = cu$ for some $c\in K$. Say $\ell_u$ be a form in degree $i$ represented by this eigenvector $u$ in $A_i$ and not in the annihilator space $V$, thus
$ g\ell_u = hc\ell_u$.  Then there is a quadratic form $g' := g-ch \in Kg + Kh - \{0\}$ such that  $g'$ annihilated by the space $V$ 
and also by $\ell_u \in A_i  \setminus V$. Hence $\dim_K \Ann_{A_i}(g') \geq \dim_K V +1$.
\end{proof}

From now on, $I = (f_1,\,\ldots,\,f_n,g,h) = \ci + (g, h)$ is a homogeneous ideal  where 
$\dim_K \big( \ci_3 \cap g'R_1 \big) \neq 2$ for a quadratic form $g' \in Kg + Kh - \{0\}$,  which means that
$ \dim_K g'A_1 \neq n-2$. Therefore $\dim_K g'A_1$ is either $n$ or $n-1$.  

\begin{proposition}\label{Prop-gA1=hA1} For the graded Gorenstein Artin $K$-algebra $A$,
if $gA_1=hA_1$ with $\dim_K gA_1=n-1=\dim_K hA_1$, that is $\dim_K(\ci_3 \cap gR_1)=\dim_K(\ci_3 \cap hR_1)  =1$, then $\EGH_{\sec2,n}(2)$ holds for the homogeneous defect two quadratic ideal $I =\ci + (g,h)$. 
\end{proposition}

\begin{proof} Since $\dim_K \Ann_{A_1}(g)= \dim_K \Ann_{A_1}(h)=1$ there is some $g' \in Kg+Kh -\{0\}$ with $\dim_K \Ann_{A_i}(g') =2$ by Lemma \ref{lemma2}. In consequence, $\dim_K \big( \ci_3 \cap g'R_1\big) = 2$, and so we are done by Proposition \ref{defect2-chen}.
\end{proof}

\begin{proposition}\label{n-dim-space}  For the graded Gorenstein Artin $K$-algebra $A$, if $\dim_K  gA_1 =\dim_K hA_1 = n$,
then there exists a quadratic form $g'$ in $Kg+Kh$ with a nonzero linear annihilator in $A$.
\end{proposition}

\begin{proof} 
By assumption $\dim_K  A_1 = \dim_K  gA_1 =\dim_K hA_1 = n$, and so we may consider again the multiplication maps $\phi_g : A_1 \to gA_1$ and $\phi_h: A_1\to hA_1$. Then we obtain a automorphism $T: A_1 \to A_1$ and there exists an nonzero linear form $\ell \in A_1$ such that $T(\ell) = c\ell$ for some $c\in K$, that is $g\ell = ch\ell$. Consider $g'= g-ch \in Kg+Kh$. Clearly, $\ell \in \Ann_{A_1}(g')$. 
\end{proof}

Next we assume that there is a linear annihilator $L\in A_1$ of $g$ where $Lh \neq 0$ over the Gorenstein ring $A=R/\ci$. This case may come up either when $\dim_K gA_1 = \dim_K hA_1 = n-1$ and the linear annihilator spaces 
$\Ann_{A_1}(g)$ and $\Ann_{A_1}(h)$ are distinct, or when  $\dim_K gA_1 = n-1$ and $\dim_K hA_1=n$. 

We shall make repeated use of the following result, which is Lemma 3.3 of  Chen's paper \cite{Chen}.  

\begin{lemma}[Chen \cite{Chen}]\label{Chen3.3}  If $\vect f n$ is a regular sequence of 2-forms in $R$ and  we have
a relation
$u_1f_1 + u_2f_2 + \cdots+ u_nf_n = 0$ for some 
$t$-forms $\vect u n$, then $\vect u n \in  (\vect fn)_t$. More precisely, we have that $t \geq 2$ and there exists a skew-symmetric  $n \times n$ matrix $B$ of $(t - 2)$-forms such that
$(u_1\, u_2\,  \cdots \, u_n) =  (f_1\, f_2 \,\cdots\, f_n)B$.  \qed \end{lemma}

\begin{proposition}\label{distinct-linear-ann}
 Let $I = \ci+\fg$ be a defect $\delta$, where $2 \leq \delta \leq n-1$, quadratic ideal of $R$ as in Notation~\ref{notfg}. If there is a linear form $L$ in $\Ann_A(g_1,\,\ldots,\,g_{\delta-1})$ such that $Lg_{\delta} \neq 0$ in $A$, then
\[ \dim_K\big( (\vect fn, g_1,\,\ldots,\,g_{\delta-1})_3 \cap g_{\delta}R_1 \big) \leq 3\]
\end{proposition}
\vspace*{2ex}

Chen \cite{Chen} used an argument involving the Koszul relations on $(x_1,\,\ldots,\,x_r)$ for $r\leq n$ while introducing another proof for Theorem \ref{EGH(d)-ci}. In the proof of this proposition we use a very similar argument. 

\begin{proof} As in Notation~\ref{notfg}, let $J = \ci+(\vect g {\delta-1})$, and denote the row vector of the regular sequence $\vect fn$ by $\vec{\bf f}$ and the row vector of quadratic forms $g_1,\,\ldots,\,g_{\delta-1}$ by  $\vec{\bf g}$.

\vspace*{2ex}

Suppose $\dim_K (J_3 \cap g_{\delta}R_1 )\geq 4$, and without loss of generality we may assume that 
\begin{eqnarray*} x_1 g_{\delta} &= \vec{\bf g} \cdot \vec{\ell_1} + \vec{\bf f}\cdot \vec{p_1} \\ 
                           x_2 g_{\delta} &= \vec{\bf g} \cdot \vec{\ell_2} + \vec{\bf f}\cdot \vec{p_2}  \\
                           x_3 g_{\delta} &= \vec{\bf g}\cdot \vec{\ell_3} + \vec{\bf f}\cdot \vec{p_3} \\
                           x_4 g_{\delta} &= \vec{\bf g} \cdot \vec{\ell_4} + \vec{\bf f}\cdot \vec{p_4} 
\end{eqnarray*}
where $\vec{\ell_i}$ and $\vec{p_i}$ are column vectors of linear forms of lengths $\delta-1$ and $n$, respectively. 

We assume that there is a linear form $L$ such that $Lg_i =0$ for each $i=1,\,\ldots,\,\delta-1$ but  $Lg_{\delta} \neq 0$ in $A$.  Then we get am $n\times (\delta-1)$ matrix $(q_{i,j}) = \begin{pmatrix}\vec{q_1} & \vec{q_2} & \cdots & \vec{q}_{\delta-1}\end{pmatrix} $ of linear forms such that
                               \[ L\vec{\bf g} = \vec{\bf f}\cdot (q_{i,j}). \] 

We observe that each $x_i Lg_{\delta} $ is in $\ci$, and write $x_i Lg_{\delta} =  \vec{\bf f}\cdot \vec{Q_i}$ where $\vec{Q_i}$ is a column of quadratic forms for $i = 1,2,3,4$.  Therefore:
\begin{eqnarray}\label{step1}
Lg_{\delta} \begin{pmatrix}  x_1  &  x_2  &  x_3 & x_4 \end{pmatrix} = \vec{\bf f}\cdot \begin{pmatrix}  \vec{Q_1}  & \vec{Q_2}  &  \vec{Q_2} & \vec{Q_4} \end{pmatrix}.
\end{eqnarray}

\vspace*{2ex}

Let $M_1 = \begin{pmatrix}
x_2 & x_3 & x_4 & 0 & 0 & 0 \\
-x_1 & 0 & 0 & x_3 & x_4 & 0 \\
0 & -x_1 & 0 & -x_2 & 0 & x_4 \\
0  & 0 & -x_1 & 0 & -x_2 & -x_3
\end{pmatrix}$.  Note that $\begin{pmatrix}  x_1  &  x_2  &  \cdots & x_4\end{pmatrix}\cdot M_1 = 0$. Multiplying the equation (\ref{step1})
by $M_1$ from right gives that   $\vec{\bf f}\cdot(\vec Q_1\ \vec Q_2\ vec Q_3\ \vec Q_4) = 0$, and so all entries are 0 in
\begin{eqnarray*}
\vec{\bf f}  \begin{pmatrix}  x_2\vec{Q_1}-x_1\vec{Q_2}  &  x_3\vec{Q_1}-x_1\vec{Q_3}  & x_4\vec{Q_1}-x_1\vec{Q_4}  & x_3\vec{Q_2}-x_2\vec{Q_3}  & x_4\vec{Q_2}-x_2\vec{Q_4}  & x_4\vec{Q_3}-x_3\vec{Q_4}\end{pmatrix}
 \end{eqnarray*}
 
By Lemma~\ref{Chen3.3}, there are alternating $n\times n$ matrices $B_{12}, B_{13}, B_{14}, B_{23}, B_{24}, B_{34}$ of \underline{linear} forms such that
\begin{eqnarray} \label{step2}
 \Big( \underbrace{x_2\vec{Q_1}-x_1\vec{Q_2}}_{\tiny\begin{matrix}\text{a column vector} \\ \text{ of cubic forms}\end{matrix}} \, \,  \cdots \, \,  x_4\vec{Q_3}-x_3\vec{Q_4}\Big) = \begin{pmatrix}
 B_{12}\vec{\bf f}^T & \cdots & B_{34}\vec{\bf f}^T  \end{pmatrix}  \end{eqnarray}

Similarly,  consider the matrix $M_2 = \begin{pmatrix}
x_3 & x_4  & 0 & 0 \\
-x_2 & 0 & x_4  & 0 \\
0 & -x_2 & -x_3 & 0\\
x_1 & 0 &  0 & x_4 \\
0  & x_1 & 0 & -x_3 \\
0 & 0 & x_1 & x_2
\end{pmatrix}$
 such that  $M_1\cdot M_2 = {\bf 0}$ and multiply equation (\ref{step2}) by $M_2$ from right to obtain:
 \[\Big(
 \underbrace{(x_3B_{12}-x_2B_{13}+x_1B_{23})}_{\tiny\begin{matrix} n\times n \, \text{matrix of} \\ \text{quadratic forms} \end{matrix}}\vec{\bf f}^T  \, \, \cdots \, \,  (x_4B_{23}-x_3B_{24}+x_2B_{34})\vec{\bf f}^T 
 \Big) ={\bf 0}.\]
Then again by Lemma~\ref{Chen3.3}, there are alternating $n\times n$ matrices 
$$C^{123}_1,\,\ldots,\,C^{123}_n,\,C^{124}_1,\,\ldots,\,C^{124}_n,\,  \ldots,\,C^{234}_1,\,\ldots,\,C^{234}_n$$ of \underline{scalars} such that  

\begin{align}
\label{step 3}
\begin{split}
 x_3B_{12}-x_2B_{13}+x_1B_{23} = &  \begin{pmatrix} \vec{\bf f}C^{123}_1 \\ \vdots \\ \vec{\bf f}C^{123}_n \end{pmatrix}\\
 x_4B_{12}-x_2B_{14}+x_1B_{24} = &  \begin{pmatrix} \vec{\bf f}C^{124}_1 \\ \vdots \\ \vec{\bf f}C^{124}_n \end{pmatrix}\\
x_4B_{13}-x_3B_{14}+x_1B_{34} = &  \begin{pmatrix} \vec{\bf f}C^{134}_1 \\ \vdots \\ \vec{\bf f}C^{134}_n \end{pmatrix}\\
   x_4B_{23}-x_3B_{24}+x_2B_{34} = &\begin{pmatrix} \vec{\bf f}C^{234}_1 \\ \vdots \\ \vec{\bf f} C^{234}_n \end{pmatrix}
\end{split}   
\end{align}
Repeating the previous steps with $M_3 = \begin{pmatrix} x_4\\-x_3 \\ x_2 \\ -x_1\end{pmatrix}$, so that $M_2\cdot M_3 = {\bf 0}$, we get
\[ {\bf 0} = \begin{pmatrix}  B_{12}& B_{13} & B_{14} & B_{23} & B_{24} & B_{34} \end{pmatrix} M_2 M_3 =  \begin{pmatrix}
\vec{\bf f} C^{123}_1 & \vec{\bf f} C^{124}_1 & \vec{\bf f} C^{134}_1  & \vec{\bf f}C^{234}_1 \\
\vdots                       &         \vdots                &            \vdots              &       \vdots                 \\
  \vec{\bf f} C^{123}_n & \vec{\bf f}C^{124}_n & \vec{\bf f}C^{134}_n  & \vec{\bf f} C^{234}_n 
\end{pmatrix} M_3\]
and then for all $i = 1,2,\,\ldots,\,n$ we obtain
\[ \vec{\bf f}( x_4C^{123}_i - x_3 C^{124}_i +x_2C^{134}_i - x_1C^{234}_i) = 0. \]
Then, finally, $x_4C^{123}_i - x_3 C^{124}_i +x_2C^{134}_i - x_1C^{234}_i = 0 $ for all $i = 1,2,...,n$. Hence, 
\[C^{123}_i = C^{124}_i = C^{134}_i = C^{234}_i = 0 \,\, \text{ for all } \,\, i = 1,2,...,n.\]

Thus, in (\ref{step 3}) we get $x_3B_{12}-x_2B_{13}+x_1B_{23} = 0 $. This shows that $x_3$ divides every entry in $x_2B_{13}-x_1B_{23}$. Therefore we may rewrite $B_{13} = x_3\widetilde{B_{13}} + D_{13}$ and $B_{23} = x_3\widetilde{B_{23}} + D_{23}$, where $\widetilde{B_{13}}$ and  $\widetilde{B_{23}}$ are alternating matrices of scalars,  $D_{13}$ and  $D_{23}$ are  alternating matrices of linear forms that do not contain $x_3$, and 
$x_2D_{13} - x_1D_{23} =0$.  We obtain the following 
\[B_{12} =\frac{1}{x_3}(x_2B_{13}-x_1B_{23}) = x_2\widetilde{B_{13}}-x_1\widetilde{B_{23}}\]
Returning to equation (\ref{step2}), we obtain 
$ x_2 \vec{Q_1} - x_1\vec{Q_2} = B_{12} \vec{\bf f}^T = (x_2\widetilde{B_{13}}-x_1\widetilde{B_{23}}) \vec{\bf f}^T$. 
Consequently,
\[x_1 \big(  \vec{Q_2} - \widetilde{B_{23}}\vec{\bf f}^T\big) = x_2 \big(  \vec{Q_1} - \widetilde{B_{13}} \vec{\bf f}^T\big)  \]
which tells us that $x_1$ divides  every entry of $\vec{Q_1} - \widetilde{B_{13}}\vec{\bf f}^T$.  It follows that

\begin{align*}
 \vec{\bf f}\big( \vec{Q_1} - \widetilde{B_{13}}\vec{\bf f}^T\big)  &=  \vec{\bf f}\vec{Q_1}  \hspace*{4ex} \text{ as  $\widetilde{B_{13}}$ is alternating and  $\vec{\bf f}\widetilde{B_{13}}\vec{\bf f}^T =0$} \\
 &= x_1 Lg_{\delta}  \hspace*{4ex} \text{ by equation } (\ref{step1}).
\end{align*}
This shows that  $Lg_{\delta} = \vec{\bf f} \frac{1}{x_1}\Big(\vec{Q_1} -\widetilde{B_{13}}\vec{\bf f}^T \Big) \in 
(\vect fn )_3$, which contradicts our assumption $L \notin \Ann_{A}(g_{\delta})$.
\end{proof}
%
%
%
%
\begin{corollary}\label{EGH2n2} Let $I = \ci+\fg \subseteq R$ be a defect $\delta$ quadratic ideal with $2\leq\delta\leq n-1$.
Suppose that $$(\dagger) \quad  \Ann_{A_1}(g_1,\,\ldots,\,g_{\delta-1}, g_\delta) \setminus \Ann_{A_1}(g_\delta) \neq \emptyset.$$
Then 
\[\dim_K I_3 \geq  \dim_K \lpp_3 \]
where $\lpp = (x_1^2,\,\ldots,\,x_n^2) + ( x_1x_2, x_1x_3,\,\ldots,\,x_1x_{\delta+1})$ is the defect $\delta$ lex-plus-powers ideal of $R$.
\vspace*{2ex}
That is, $\EGH_{{\bf 2}, n}(2)$ holds for any defect $\delta$ quadratic ideal with property $(\dagger)$.
 \end{corollary}

 \begin{proof}  Notice that $\dim_K \mathcal{L}_3  = n^2+ n\delta -\frac{\delta(\delta+3)}{2}$.
 We use induction on $\delta$. Let $J= \ci+(g_1,\,\ldots,\,g_{\delta-1})$ be the defect $\delta-1$ quadratic ideal.
 
\begin{align*}\dim_K I_3 &= \dim_K J_3 + n - \dim_K\big( J_3 \cap g_\delta R_1 \big)  \\
& \geq \Big( n^2+(\delta-1)n - \frac{(\delta-1)(\delta+2)}{2} \Big) + n - 3  =  n^2+ n\delta -\frac{\delta(\delta+3)}{2} +\delta-3  \\
& \geq n^2+ n\delta -\frac{(\delta)(\delta+3)}{2}.
\end{align*}
\end{proof}

We notice that a special case of Corollary~\ref{EGH2n2} when $\delta=2$ shows that the inequality is strict.
\begin{corollary} \label{Cor-diff-linear-ann}Let $I = \ci+ (g,h) $ be a defect two ideal generated by quadrics in $R$. 
If $\Ann_{A_1}(g) = \Span\{L\}$ for some $L\in R_1$ and  $L$ does not annihilate $h$ in $A = R/\ci$, then 
$$\dim_{K} I _3 \geq n^2+2n-4 > \dim_{K}(x_1^2,\,\ldots,\,x_n^2, x_1x_2, x_1x_3)_3$$
\end{corollary} 

\begin{proof} The result follows from Proposition \ref{distinct-linear-ann} as
\[\dim_{K} I _3  = n^2+2n- \underbrace{\dim_K (\ci_3 \cap gR_1)}_{=\dim_K \Ann_{A_1}(g)=1}  - \underbrace{\dim_K (J_3 \cap hR_1 )}_{\leq 3}\] 
which is $\geq n^2+2n-4$.

\end{proof}

Finally, we give an affirmative answer to the Main Question \ref{main-question}.

\begin{theorem}\label{thm-defect2-egh(2)} Let $I = \ci+ (g,h)\subseteq R=K[x_1,\,\ldots,\,x_n]$ for $n\geq 5$ be a defect two ideal quadratic ideal. Then
\[\dim_K I_3 \geq n^2+2n-5.\]

More precisely, $\EGH_{\sec2,n}(2)$ holds for homogeneous defect two quadratic ideals in $R$ for any $n\geq 5$. 
\end{theorem} 

\begin{proof}
If the given defect two ideal satisfies Proposition \ref{defect2-chen} , then, by Chen's result, the theorem is proved.  

Assume that $\dim_K \big( \ci_3 \cap g'R_1\big) \neq 2$ for any $g'\in Kg+Kh \setminus \{0\}$. If $\dim_K \big(\ci_3 \cap gR_1\big) = \dim_K\big(\ci_3 \cap hR_1\big)=0$, by 
Proposition \ref{n-dim-space}, we can always find another quadratic form $g' \in Kg+Kh \setminus \{0\} $ so that $g'$ has a linear annihilator in $A$. Then we can apply Corollary \ref{Cor-diff-linear-ann}. If $\dim_K \big(\ci_3 \cap gR_1\big) = \dim_K\big(\ci_3 \cap hR_1\big)=1$ and the same linear form annihilates both $g$ and $h$ in $A$, by Proposition 
\ref{Prop-gA1=hA1}.
we have a situation contradicts our assumption. \end{proof}

\begin{corollary}\label{cor-EGH(2)}$\EGH_{\sec2,n}(2)$ holds for every defect two ideal containing a regular sequence of quadratic forms.
\end{corollary}

\begin{proof} This result follows from Lemma \ref{equal-EGH(d)} and Theorem \ref{thm-defect2-egh(2)}.
\end{proof}

%
%
%
%
%
%
%
%
%
%
%
%
%
%

\section{The $\EGH$ conjecture when $n=5$ and $a_1=\cdots=a_5=2$}

In this section  $ R= K[x_1,\ldots,\,x_5]$ and $I = (f_1,\,\ldots,\,f_5)+( g_1,\,\ldots,\,g_\delta) = \ci + \fg$ is a homogeneous defect $\delta$ ideal in $R$, where  $f_1,\,\ldots,\,f_5$ is a regular sequence of quadrics and $\deg g_j \geq 2$ for $j=1,\,\ldots,\,\delta$. Throughout, we shall write $A := R/\ci$,  which is a graded Gorenstein local Artin ring.  We will show the existence of a 
lex-plus-powers ideal $\mathcal L \subseteq R$ containing $x_i^2$ for $i=1,\,\ldots,\,5$ with the same Hilbert function as $I$ by proving the following main theorem.
\begin{theorem} \label{thm-EGH-n=5} The EGH conjecture holds for all homogeneous ideals containing a regular sequence of quadrics in $K[x_1,\,\ldots,\,x_5]$.
\end{theorem}

Lemma \ref{CM-lemma} of Caviglia-Maclagan tells us that $\EGH_{\sec2,5}(d)$ holds if and only if $\EGH_{\sec2,5}(5-d-1)$ holds. Thus it will be enough to show $\EGH_{\sec2,5}(d)$ when $d=0,1,2$. 
By Remark \ref{Rmk-EGH-01} we know that $\EGH_{\sec2,5}(d)$ is true when $d=0,1$, therefore  $\EGH_{\sec2,5}(3)$ and $\EGH_{\sec2,5}(4)$ both hold as well. 

Our goal in this section is to prove $\EGH_{\sec2,5}(2)$ for any homogeneous ideal containing a regular sequence of quadrics: this will complete the proof of $\EGH_{\sec2, 5}$.  To achieve this, it suffices to understand $\EGH_{\sec2,5}(2)$ for quadratic ideals with arbitrary defect $\delta$ (but, of course, $\delta \leq 10$,  since $\dim_K R_2 = 15$), by Lemma \ref{equal-EGH(d)}.

\begin{remark}\label{def2-quad-n5} As a result of Corollary \ref{cor-EGH(2)}, 
we see that $\EGH_{\sec2,n}$ holds for any defect $\delta=2$ quadratic ideal in 
$K[x_1,\,\ldots,\,x_n]$ for  $n=5$.
\end{remark} 

To accomplish our goal we will prove $\EGH_{\sec2,5}(2)$ for defect $\delta\geq 3$ quadratic ideals.  In the next
subsection, we prove that if one knows the case where $\delta = 3$,  on obtains all the cases for $\delta \geq 4$.  
In the final subsection we finish the proof by establishing $\EGH_{\sec2,5}(2)$ for $\delta = 3$.  

%
%
%
\subsection*{Quadratic ideals with defect $\delta \geq 4$}
\begin{lemma}\label{lemma_higher_defect}  If $\EGH_{\sec2,5}(2)$ holds for all defect three quadratic ideals, then it holds for all  quadratic ideals with defect $\delta \geq 4$.
\end{lemma}

\begin{proof} 
Let $I = (f_1,\,\ldots,\,f_5, g_1, g_2, g_3, g_4)=\ci+\fg \subseteq R$ be a defect $4$ homogeneous ideal generated by quadrics, where $f_1,\,\ldots,\,f_5$ form a regular sequence. By assumption the defect three quadratic ideal $J = \ci+( g_1, g_2, g_3) \subseteq I$ satisfies $\EGH_{\sec2,5}(2)$, that is, $\dim_K J_3 \geq 31$.

Let $\lpp = (x_1^2,\,\ldots,\,x_5^2, x_1x_2, x_1x_3,x_1x_4, x_1x_5)$ be the LPP ideal with $\dim_K \lpp_2 =\dim_K I_2 =  9$.  Then we get $\dim_K I_3 \geq \dim_K J_3 \geq 31= \dim_K \lpp_3$, as we need for the case of defect $\delta=4$.

Now assume $5 \leq \delta \leq 10$. Let $\dI$ denote an arbitrary defect $\delta$ quadratic ideal, and 
let $\dL$ denote the lex-plus-power ideal with defect $\delta\geq 5$. More precisely, 
$\dL:= (x_1^2,\,\ldots,\,x_5^2) + ( m_1,\,\ldots,\,m_{\delta})$ where $m_i$ are the next greatest quadratic square-free monomials with respect to lexicographic order. 
We need to show that $ \Hf_{R/\dI}(3) \leq  \Hf_{R/\dL}(3)$. 

We assume that $\Hf_{R/\dI}(3) \geq  \Hf_{R/\dL}(3)+1$, and we shall obtain a contradiction.

Using duality for Gorenstein rings, we know that for $0\leq d \leq 5$
we have that
$$ \Hf_{R/\dI}(d) =  \Hf_{R/ \ci}(d) -  \Hf_{R / ( \ci :\dI )} (5-d).$$

Then, for $d=3$, using the assumption we get

\begin{eqnarray*}
 \Hf_{R /( \ci : \dI )}(2) &=&   \Hf_{R/ \ci}(3) -  \Hf_{R/\dI}(3) \\
&\leq & 9- \Hf_{R/\dL}(3) = \begin{cases} 7  \,  \text{ if } \, \delta=5, \\ 
8  \, \text{ if } \, \delta=6,7,\\ 
9  \, \text{ if } \, \delta=8,9,10. \\ \end{cases}
\end{eqnarray*}

We next show that $ \dim_K(\ci : \dI)_{1}=0$. If there is a nonzero linear form $\ell \in \ci : \dI$ then $\dim_K \Ann_{A_2}\ell A\geq \delta\geq 5$, so we get that $\dim_K A_3/\ell A_2 \geq 5$. On the other hand, we see that 
$A_3/\ell A_2 \cong [R/(\bar{f_1},\,\ldots,\,\bar{f_4},\bar{f_5},l)]_3$ where the $\bar{f_i}$ are the images of the $f_i$,
and the dimension of $[R/(\bar{f_1},\,\ldots,\,\bar{f_4},\bar{f_5},l)]_3$ as a $K$-vector space is at most $4$. 

Then we can find a defect $\gamma$ quadratic ideal $\gJ \subseteq  \ci : \dI$ for $\gamma=3, 2, 1$ if the defect of $\dI$ is 
$\delta=5$ or $\delta=6,7$ or $\delta=8,9,10$, respectively. We then have the inequalities shown below, where the first
is obvious  and the second follows by comparison with Hilbert functions of quotients by LPP ideals in degree 3 
and the fact that, by assumption, EGH$_{\sec2,5}(2)$ holds for quadratic ideals with defect less than or equal to three. \begin{eqnarray*}
 \Hf_{R/( \ci : \dI )}(3) &\leq&  \Hf_{R/\gJ }(3)  \leq \begin{cases} 4  \,  \text{ if $\gJ$ is a defect $\gamma=3$ quadratic ideal when $\delta=5$} , \\ 
5  \, \text{ if $\gJ$ is a defect $\gamma=2$ quadratic ideal when $6\leq \delta \leq 7$} ,\\ 
7  \, \text{ if $\gJ$ is a defect $\gamma=1$ quadratic ideal  when $8\leq \delta \leq 10$}. \\ \end{cases} 
\end{eqnarray*}

 However, each of the cases above contradicts the following equality:

$$  \Hf_{R/( \ci :\dI )}(3) =   \Hf_{R/ \ci}(2) -  \Hf_{R/\dI}(2) = \delta.$$

Thus, we get $ \Hf_{R/\dI}(3) \leq  \Hf_{R/\dL}(3)$ for any defect $\delta\geq 5$ quadratic ideal $\dI$ in $R$.
\end{proof}

%
%
%
\subsection*{Defect three quadratic ideals}

\begin{lemma} Let $I = \ci+(g_1,g_2,g_3)$ be a defect three quadratic ideal in the polynomial ring $R$. Then, for any $1\leq i_1 <  i_2 \leq 3$,
\[ \dim_K ( \ci : ( g_{i_1}, g_{i_2}) )_1 \leq 1,\]
and, furthermore, $ \dim_K ( \ci : ( g_1, g_2, g_3) )_1 \leq 1$.
\end{lemma}
\begin{proof}  Suppose that $\dim_K (\ci : ( g_1, g_2) )_1\geq 2$, and assume there are $\ell_1, \ell_2 \in R_1$ such that $\ell_i g_1, \ell_ig_2 \in  \ci$ for both $i=1,2$. Without loss of generality we assume that $\ell_1 =x_1$ and $\ell_2 = x_2$. 

Therefore, we can write $(x_1, x_2, f_1,\,\ldots,\,f_5) \subseteq  \ci : I $.  Then
\begin{align*}
2& = \Hf_{ (f_1,\,\ldots,\,f_5, g_1, g_2) / \ci}(2)  = \Hf_{R \big/ \big(\ci  :  (f_1,\,\ldots,\,f_5, g_1, g_2)\big)}(5-2 ) , \hspace*{1ex} (\text{by duality})\\ 
 &\leq \Hf_{R/(x_1, x_2, f_1,\,\ldots,\,f_5)}(3) \\
  &= \Hf_{ K[x_3, x_4, x_5]/(\bar{f_1},\,\ldots,\,\bar{f_5})}(3),  \hspace*{1ex} (\text{where $\bar{f_i}$ is the image of $f_i$ in $K[x_3, x_4, x_5]$,})\\
&\leq  \binom{5-2}{3} = 1,
\end{align*}
which is a contradiction. 
\end{proof}  

Hence, working in the graded Gorenstein Artin $K$-algebra $A= R/\ci$, we have from the lemma just above that 
$\Ann_{A_1} (g_1, g_2)$ is a $K$-vector space of dimension at most one, and, therefore
\[\dim_K \Ann_{A_1} (g_1, g_2, g_3) \leq 1\] 
since $\Ann_{A_1} (g_1, g_2, g_3) \subseteq \Ann_{A_1} (g_1, g_2)$. 

\begin{remark} \label{def3-rmk1} By Remark \ref{def2-quad-n5} we know that for any defect two quadratic ideal $J$  in $R$, $\dim_K J_3 $ is at least $30$. Then $\EGH_{\sec2,5}(2)$ holds for the defect three quadratic ideals $I$ containing a defect two quadratic ideal $J$ with $\dim_KJ_3 \geq 31$, as $ \Hf_{R/I}(3) \leq  \Hf_{R/J}\leq 4$.  
\end{remark}

We henceforth focus on defect three quadratic ideals $I = \ci+(g_1,g_2,g_3)$
 in $ R$ such that every defect two quadratic ideal $J \subseteq I$ containing $\ci$ has $\dim_K J_3 = 30$.

For such defect three quadratic ideals, we observe the following.

\begin{remark}\label{def3-rmk2} Consider the ideal $\mathcal I = (g_1,g_2,g_3)A$ in the Gorenstein ring $A$ such that any ideal $(g_{i_1}, g_{i_2})A $ contained in $\mathcal I$ has degree three component of dimension $\dim_K(g_{i_1}, g_{i_2})A_1=5$. 
Assuming that $\dim_K\Ann_{A_1}(g_1)=1$, we have that $\Ann_{A_1}(g_1,g_2, g_3) = \Ann_{A_1}(g_1)$.

Furthermore, if $g_1A_1$ is $5$-dimensional, that is, there is no linear form that annihilates $g_1$ in $A$, then for any quadric $g$ in $Kg_1+Kg_2+Kg_3$ the vector space $gA_1 \subseteq A_3$ is either $3$ or $5$ dimensional. 
 \end{remark}
 
\begin{proof}Let $\dim_K \Ann_{A_1}(g_1)=1$, and let the linear form $L$ annihilate $g_1$ but not some form $g' \in Kg_2+Kg_3 $ in $A$. We define a defect two quadratic ideal  
$$J = (f_1,\,\ldots,\,f_5, g_1, g') \subseteq \ci+(g_1,g_2,g_3)$$ 
in $R$. Hence,  by Corollary \ref{Cor-diff-linear-ann}, we know already that $\dim_K\,J_3 \geq 31$, which means that  
$\dim_K(g_1, g')A_1=6$. This contradicts our assumption. Thus, $L$ must be in $\Ann_{A_1}(g_1,g_2,g_3)$. 
\end{proof}

Recall that the following holds, by Proposition~\ref{distinct-linear-ann}, when $\delta = 3$.
\begin{proposition}\label{def3-prop1} Let $I = \ci+( g_1,g_2,g_3) \subseteq K[x_1,...,x_5]$ be a defect $3$ quadratic ideal.

As usual, let $A = R/\ci$.   If there is a linear form $L \in \Ann_{A} (g_1,g_2)$ such that $L \notin \Ann_{A}(g_3)$, then 
\[\dim_K\big( (\ci+(g_1,g_2))_3 \cap g_3R_1 \big) \leq 3.\] \qed
\end{proposition}
When a defect three quadratic ideal $I$ satisfies the condition of the above proposition, we notice a sharp bound for 
$\Hf_{R/I}(3)$.
\begin{corollary}\label{def3-cor3} Given a defect three quadratic ideal $I = \ci+(g_1,g_2,g_3)$ in $R=K[x_1,\,\ldots,\,x_5]$, and, as usual, let $A=R/\ci$, which is a graded Gorenstein Artin ring.  If $\dim_K \Ann_{A_1}(g_1,g_2) =1$ and 
$\Ann_{A_1}(g_1,g_2, g_3) = 0$  then 
$$\dim_K I_3 \geq 32 > \dim_K \lpp_3,$$
where $\lpp = (x_1^2,\,\ldots,\,x_5^2, x_1x_2, x_1x_3, x_1x_4)$.
\end{corollary}

\begin{proof} By assumption there is a linear form in $\Ann_{A}(g_1,g_2)$, say $L$, such that $L$ does not annihilate $g_3$. Hence, Proposition \ref{def3-prop1} gives us $\dim_K\big( (\ci +(g_1,g_2))_3 \cap g_3R_1 \big)\leq 3$.
Then we get
\begin{eqnarray*}
\dim_K(\ci +( g_1, g_2,g_3))_3 &=& \dim_K(\ci +(g_1,g_2))_3+ \dim_K g_3R_1 \\
 & & -\dim_K\big( (\ci +(g_1,g_2))_3 \cap g_3R_1 \big)  \\
&\geq& 30+5-3 = 32 > 31=\dim_K \lpp_3.
\end{eqnarray*}
\end{proof}
%
%
%
\begin{proposition} \label{def3-prop2} Suppose that for all quadratic forms $g$ in $Kg_1+Kg_2$, the subspace $gA_1$ of $A_3$ is a $3$-dimensional. If $\dim_K(g_1,g_2)A_1 = 5$, then $\dim_K \Ann_{A_1}(g_1, g_2) = 1$.
\end{proposition}

 We first state the following observation in a linear algebra setting, which will be useful for the proof Proposition \ref{def3-prop2}.
\begin{lemma}\label{lin_alg} Let $S$, $T$ be linear transformations from $V$ to $W$, both $n$-dimensional vector spaces over $K$, such
that $\rank(S) = \rank(T)= \rank(S-T) = r$, and the kernels of $S$, $T$ are disjoint. Then the images of $S$ and $T$ are contained in the same $(3r-n)$-dimensional subspace of $W$.
\end{lemma}
\begin{proof} $V_0 = \ker(S - T)$ is $(n-r)$-dimensional.  $S$ and $T$ are injective on $V_0$, since for
$v \in V_0$,  $S(v) = 0$
iff $T(v) = 0$, and $\Ker(S) \cap \Ker(T) = 0$.  
Thus, $S(V_0) = T(V_0)$ is an $(n-r)$-dimensional space in $S(V) \cap T(V)$.
Since $S(V), T(V)$ are $r$-dimensional and overlap in a space of dimension at least $n-r$, $S(V) + T(V)$
has dimension at most $r + r -(n-r) = 3r-n$.
\end{proof}

\begin{proof}[Proof of Proposition \ref{def3-prop2}] Assume that $\dim_K \Ann_{A_1}(g_1, g_2) = 0$.
Since all quadratic forms $g$ in $Kg_1+Kg_2$ are such that $gA_1 \subseteq A_3$ has vector space dimension 3, we have from 
Lemma \ref{lin_alg} with $n=5, r=3$, that $(Kg_1 +Kg_2)A_1 \subseteq A_3$ is at most $4$-dimensional. Consequently, 
$$\dim_K[ A/(g_1,g_2)A]_3 = \dim_K [R/\ci+( g_1,g_2)]_3 \geq 6,$$ contradicting $\EGH_{\sec2,5}(2)$ for  defect 2 quadratic ideals. Hence,  $\dim_K \Ann_{A_1}(g_1, g_2) = 1$.
\end{proof}

\begin{proposition}\label{prop_def3_dimAnn = 0} Let $I=\ci+(g_1,g_2, g_3)$ be a defect three quadratic ideal in $R = K[x_1,\,\ldots,\,x_5]$.   If $\dim_K \Ann_{A_1}(g_1,g_2,g_3) = 0$ then $\Hf_{R/I}(3) \leq 4$.
\end{proposition}

\begin{proof} 

First, by Remark \ref{def3-rmk1} we note that it suffices to consider any defect two quadratic ideal $J\subseteq I$ with $ \Hf_{R/J}(3)=5$.

Suppose that $\dim_K \Ann_{A_1}(g_1,g_2,g_3) = 0$.  Then, clearly, no $g_i$,  for $i=1,2,3$ has a $1$-dimensional linear annihilator space in $A$, since, otherwise, by Remark \ref{def3-rmk2}, we obtain that $\dim_K \Ann_{A_1}(g_1,g_2,g_3)=1$, which contradicts our assumption. Thus, for the rest of the proof we may assume that each $g_iA_1$, $i=1,2,3$, is either $3$ or $5$ dimensional. 

If all forms $g$ in $Kg_1+Kg_2+Kg_3$ are such that $\dim_K gA_1 = 3$ then we can find two independent quadratic forms whose linear annihilator spaces intersect in $1$-dimensional space, and the result follows from Corollary \ref{def3-cor3}. 

Let $g_1A_1$ be a $5$-dimensional subspace of $A_3$ and for every $g \in Kg_2+Kg_3$, $gA_1$ has dimension either $3$ or $5$. 

We complete the proof by obtaining a contradiction. We assume that $\Hf_{R/I}(3) = 5$. In other words, the space $W = (Kg_1 + Kg_2 + Kg_3)A_1 \subseteq A_3$ is 5-dimensional. Then we get $W= g_1A_1 = (Kg_2 +Kg_3)A_1$.

Consider the multiplication maps by $g_1, g_2$ and $g_3$ from $A_1$ to the subspace $W$ of $A_3$. By adjusting the bases of $A_1$ and $W$ we can assume the matrix of $g_1$ is the identity matrix $\mathbb{I}_5$ of size $5$. Denote the matrices of $g_2$ and $g_3$ by $\alpha$ and $\beta$, respectively. We can assume that $\alpha$ and $\beta$ are both singular, and so
have rank $3$, by subtracting the suitable multiples of $\mathbb{I}_5$ from them if they are not singular.

We see that all matrices $z\mathbb{I}_5 + x\alpha + y\beta$ must have at most two eigenvalues, otherwise we can form a linear combination whose kernel is $1$-dimensional, which corresponds to a quadratic form with $1$-dimensional linear annihilator space. Then there are two main cases: one is that every matrix in the space spanned by $\mathbb{I}_5, \alpha$ and $\beta$ has one eigenvalue. The other is that almost all matrices in the form $z\mathbb{I}_5+ x\alpha + y\beta$ have two eigenvalues, since the subset with at most one eigenvalue is Zariski closed. 

Define $D(x,y,z) = \det(z\mathbb{I}_5 - x\alpha - y\beta)$, a homogeneous polynomial in $x,y,z$ of degree $5$ that is monic in $z$. Note that  $D$ is also the characteristic polynomial, in $z$,  of $x\alpha+y\beta$.  Notice that the singular matrices in the subspace of $5\times5$ matrices spanned by $\mathbb{I}, \alpha$ and $\beta$ are defined by the vanishing of $D$. 

If the determinant $D$ is square-free (as the characteristic polynomial in $z$), then the ideal $(D)$ is a radical ideal and it cannot contain a nonzero polynomial of degree less than $5$, which contradicts the fact that all size  $4$ minors of a singular matrix must vanish, since in our situation these singular matrices have rank $3$.  Therefore the size $4$ minors, whose degrees are at most $4$, are in the radical $(D)$. 

If the determinant $D$ is not square-free,  then its squared factor must be 
linear or quadratic:  in the latter case the other factor is linear, so that in either case $D$ has a linear
factor,  say  $z - ax - by$. 

Consider the independent matrices   $\alpha' =   a\mathbb{I}_5 - \alpha$,   $\beta'  = b\mathbb{I}_5 - \beta$.  Then we think of any linear
combination of them, say  $r\alpha' + s\beta' = r(a\mathbb{I}_5 - \alpha) + s(b\mathbb{I}_5 - \beta) = (ar+bs)\mathbb{I}_5 - r\alpha  - s\beta.$
 As $z-ax-by$ is a factor of $D(x,y,z)$,  and
hence,  $D$ vanishes for $x = r,  y = s, z = ar+bs.$ This means that every linear combination of  $\alpha'$ and
$\beta'$ is singular.  Therefore, we can replace  $\alpha, \beta$  by  $\alpha'$ and  $\beta'$ and so we can assume 
that we are
in the case where every linear combination of the two non-identity matrices is singular, and, if not $0$,
of rank $3$.  By Lemma~\ref{lin_alg}, this implies that the kernels of $\alpha'$ and  $\beta'$ cannot be disjoint, so we are done by Proposition \ref{def3-prop2} and Corollary \ref{def3-cor3}.

\end{proof}

%
%
%
%
%

In order to prove $\EGH_{\sec2,5}(2)$ for every defect three quadratic ideal $I = \ci+(g_1,g_2,g_3)$ in 
$R=K[x_1,\,\ldots,\, x_5]$ we must also discuss the cases when there is a nonzero linear form 
$L \in \Ann_{A} (g_1, g_2, g_3)$.
\begin{proposition}\label{prop_def3_dimAnn = 1} Let $I = \ci+(g_1,g_2,g_3)$ be a defect three quadratic ideal in $R$. 
If $\Ann_{A_1}(g_1, g_2, g_3)$ is a $1$-dimensional $K$-subspace of $A_1$,  say $KL$,  then 
\[ \Hf_{R/I}(3) = 4.\]
\end{proposition}
\begin{claim} One of the  quadratic forms $f_i$ in the regular sequence has the linear factor $L$.
\end{claim}
\begin{proof}[Proof of claim]
As $g_1, g_2, g_3 \in \Ann_{A_2}(L) \subseteq A_2$ for $L \in \Ann_{A_1}(g_1, g_2, g_3)$ we know that $$\dim_K \Ann_{A_2}(L) \geq 3.$$ 
This tells us that $\dim_K LA_2 \leq 7,$ which implies 
\begin{eqnarray}\label{dim_least} \dim_K (A_3/LA_2)  = \dim_K[A/LA]_3 \geq 3 \end{eqnarray}
as $\dim_K A_3 = 10$.

Assume that $L=x_5$ and let $\bar{f}_i$ be the image of $f_i$ modulo $x_5$ .

Suppose that $\bar{\ci} = (\bar{f}_1, \bar{f}_2, \bar{f}_3, \bar{f}_4, \bar{f}_5)$ is an almost complete intersection in 
$K[x_1, x_2, x_3, x_4]$.
Thus, 

\[A/LA \cong \frac{K[x_1,\,\ldots,\,x_5]}{\ci+(x_5)} \cong  \frac{K[x_1, x_2, x_3, x_4]}{ \bar{\ci} }.\]

However, using the Francisco's result for almost complete intersections \cite{Fr}, we know that 
\[\dim_K \Big[  \displaystyle\frac{K[x_1, x_2, x_3, x_4]}{ \bar{\ci} } \Big]_3 \leq 2 = \dim_K \Big[  \displaystyle\frac{K[x_1, x_2, x_3, x_4]}{ (x_1^2,\,x_2^2,\,x_3^2,\,x_4^2,\,x_1x_2)} \Big]_3. \]
This contradicts (\ref{dim_least}).

Hence the images of $f_i$ modulo $L$ form a regular sequence in $K[x_1,\,\ldots,\,x_4]$, that is, one of them has a linear factor $x_5$. 
\end{proof}

As a result of the claim, after a suitable change of variables, we may assume that the linear annihilator is $L = x_5$ and may consider $I$ in two possible forms:  either $I$ is in the form of (\ref{case1-x1x5}) in \textit{Case 1} below, where $f_1,\,f_2,\,f_3,\,f_4,\,x_1x_5$ is the regular sequence, or $I$ is as in (\ref{case2-x_5x_5}) in \textit{Case 2} below, where $f_1,\,f_2,\,f_3,\,f_4,\,x_5^2$ form a quadratic regular sequence in $I$.  
\bigskip
\\ \noindent \textit{Case 1.} Suppose that $f_5 = x_1x_5$. Then we can assume that $g_1 = x_1x_2,\,g_2=x_1x_3,\,g_3=x_1x_4$.
Furthermore, after we alter the $f_i$ by getting rid of all the terms containing $x_1$ except $x_1^2$, we may assume that the defect three quadratic ideal $I$ looks like 
\begin{equation}\label{case1-x1x5}
I = (f_1,\,f_2,\,f_3,\,f_4+cx_1^2,\,x_1x_5,\,x_1x_2,\,x_1x_3,\, x_1x_4),
\end{equation}
where $f_1, f_2, f_3, f_4$ form a regular sequence in $K[x_2, x_3, x_4, x_5]$ and $c\in K$.

\begin{proposition} Let $I = (f_1,\,f_2,\,f_3,\,f_4+cx_1^2,\,x_1x_5,\,x_1x_2,\,x_1x_3,\,x_1x_4)$ 
be a defect three quadratic ideal in $R$ where $f_1, f_2, f_3, f_4$ is an $K[x_2,x_3,x_4,x_5]$-sequence. Then
\[ \Hf_{R/I}(3) = 4 = \Hf_{R/\lpp}(3)\]
where $\lpp = (x_1^2,\,\ldots,\,x_5^2,\,x_1x_2,\,x_1x_3,\,x_1x_4)$.
\end{proposition} 
\begin{proof} One can easily see that $I$  contains all cubic monomials divisible by $x_1$ since  $x_1x_i \in I$ for all $i=2,3,4,5$ and $f_4$ is a quadratic form in $K[x_2,x_3,x_4,x_5]$, therefore $x_1f_4\in I$ and so is $x_1^3$. Thus, the Hilbert functions of 
$R/I$ and $k[x_2,x_3,x_4,x_5] \big/ I\cap K[x_2,x_3,x_4,x_5]$ agrees in degree $3$. So 
 $\Hf_{R/I}(3) = \Hf_{K[x_2,x_3,x_4,x_5] \big/ I\cap K[x_2,x_3,x_4,x_5]}(3) = \Hf_{K[x_2,x_3,x_4,x_5] 
 \big/ (f_1,\,f_2,\,f_3,\,f_4)}(3) = 4$
\end{proof}
\bigskip
\noindent\textit{Case 2.} Suppose that $f_5=x_5^2$ by altering the variables and generators, and then we can assume that $g_1= x_1x_5$, $g_2=x_2x_5$, $g_3=x_3x_5$.
As we did in the case above, we get rid of all the terms containing $x_5$ except $x_4x_5$ in the $f_i$, and so the defect three quadratic ideal can be written as follows:
\begin{equation}\label{case2-x_5x_5}
I = (f_1,\,f_2,\,f_3,\,f_4+cx_4x_5,\,x_5^2,\,x_1x_5,\,x_2x_5,\,x_3x_5),
\end{equation}
 where $f_1,\,f_2,\,f_3,\,f_4$ form a regular sequence in $K[x_1, x_2, x_3, x_4]$ and $c\in K$.


\begin{lemma} Let $\mathfrak a = (f_1,\,f_2,\,f_3,\,f_4+x_4x_5,\,x_5^2):(x_1x_5,\,x_2x_5,\,x_3x_5) $ be the colon ideal in $R$. Then
we have  $\Hf_{R/\mathfrak{a}}(2) =6.$
\end{lemma}

\begin{proof} It suffices to show $\dim_K \mathfrak{a}_2  = 9$.

We know that $x_1x_5,\,x_2x_5,\,x_3x_5,\,x_4x_5,\,x_5^2$ are all in $\mathfrak{a}_2$, and $f_1,\,f_2,\,f_3,\,f_4 \in \mathfrak{a}_2$ as well.  Thus we see that $\dim_K \mathfrak{a}_2 \geq 9$. 

If there is another independent quadratic form in $\mathfrak{a}$, it must be in $R[\check{x}_5]$, as we have all quadratic monomials containing $x_5$, so call it $Q$ in $R[\check{x}_5]$. 
Then we consider the cubic form  $H = x_5Q.$ Clearly $H$ is not in the $R_1$-span of $f_1,\,f_2,\,f_3,\,f_4,\,x_5^2$, therefore we can define the ideal $J = (f_1,\,f_2,\,f_3,\,f_4,\,x_5^2,\,H)$, which is an almost complete intersection in $R$.
Then we get $\dim_K \big( (f_1,\,f_2,\,f_3,\,f_4, x_5^2)_4 \, \cap \, HR_1 \big) \geq 4$  as $x_1H, x_2H, x_3H$ and $x_5H$ are in $(f_1,\,f_2,\,f_3,\,f_4, x_5^2)_4 $, but by Corollary \ref{cor-intersection-dim} this dimension must be at most $3$. This proves that there cannot be such a quadratic form $Q$ in $\mathfrak{a}$.
\end{proof}

\begin{proposition} Let $I = (f_1,\,f_2,\,f_3,\,f_4+x_4x_5,\,x_5^2,\,x_1x_5,\,x_2x_5,\,x_3x_5)$ 
be a defect three quadratic ideal in $R$ where $f_1, f_2, f_3, f_4$ is an $R[\check{x}_5]$-sequence. Then
\[ \Hf_{R/I}(3) =4 = \Hf_{R/\lpp}(3)\]
where $\lpp = (x_1^2,\,\ldots,\,x_5^2,\,x_1x_2,\,x_1x_3,\,x_1x_4)$.
\end{proposition}

\begin{proof} Using the duality of Gorenstein algebras, again we can obtain
\[ \Hf_{R/I}(3) = \Hf_{R/(f_1,\,f_2,\,f_3,\,f_4+x_4x_5,\,x_5^2)}(3) - \Hf_{R/\mathfrak{a}}(5-3),\]
where $\mathfrak{a}$ is the colon ideal $(f_1,\,f_2,\,f_3,\,f_4+x_4x_5,\,x_5^2) : I$.

Then proof is done, since  $\Hf_{R/(f_1,\,f_2,\,f_3,\,f_4+x_4x_5,\,x_5^2)}(3)= 10$ and $\Hf_{R/\mathfrak{a}}(2)=6$ by the above lemma.
\end{proof}

%
%
%
%
%
%
%
%
%
%

\end{document}